\documentclass[12pt, twoside, a4paper, reqno]{amsart}

\usepackage{hyperref}
\usepackage{verbatim}
\usepackage{fullpage}
\usepackage{amsthm}
\usepackage{amsmath}
\usepackage{amssymb}
\usepackage{mathrsfs}
\usepackage{xcolor}
\usepackage{cancel}
\usepackage{listings}
\usepackage{graphicx}
\usepackage{blindtext}
\usepackage{enumitem}

\newtheorem{conjecture}{Conjecture}
\newtheorem*{lema}{Lemma}
\newtheorem*{theorem}{Theorem}
\newtheorem*{corollary}{Corollary}

\newtheorem{definition}{Definition}

\theoremstyle{remark}
\newtheorem{remark}{Remark}

\newcommand{\U}{\mathrm{U}}
\newcommand{\SU}{\mathrm{SU}}

\renewcommand{\Im}{\operatorname{\mathfrak{Im}}}

\begin{document}

\title{Probabilistic models for Gram's Law}

\author{C\u{a}t\u{a}lin Hanga}
\address{Department of Mathematics, University of York, York, YO10 5DD, United Kingdom}
\email{catalin.hanga@protonmail.ch}

\author{Christopher Hughes}
\address{Department of Mathematics, University of York, York, YO10 5DD, United Kingdom}
\email{christopher.hughes@york.ac.uk}

\maketitle

\begin{abstract}
Gram's Law describes a pattern that frequently occurs in the distribution of the non-trivial zeros of the Riemann zeta function along the critical line. Whenever Gram's Law holds true, it reduces the difficulty of computing the corresponding zeta zeros. In this paper, we provide a model that estimates how often this pattern occurs. The model is based on a conjecture that relates the statistical distribution of the zeta zeros to that of the eigenvalues of random unitary matrices.
\end{abstract}

\section{Gram's Law for the Riemann zeta function}

It is well-known that the zeros of the zeta function $\zeta(s)$ are of fundamental importance in number theory \cite{edwards}. The numerical computation of their precise values has a long history, dating back to  Riemann himself who calculated the first few zeros (unpublished), and continues up to the present day, with all the first $10^{13}$ zeros being currently known, as well as a couple of billion zeros after the $10^n$-th zero, for each $13\leq n\leq24$ \cite{gourdon}. When computing the zeros, there are two other functions related to $\zeta(s)$ that are widely used.

The first one is named the Riemann-Siegel theta function, defined as
\begin{align}
\theta(t) :=& \arg \left[ \pi^{-\frac{i t}{2}} \Gamma\left(\frac{1}{4}+\frac{it}{2}\right) \right] \notag \\
=&  \Im \log\Gamma\left(\frac{1}{4}+\frac{i t }{2}\right)-\frac{\log\pi}{2} \, t     \label{eq:hardytheta}
\end{align}
where $\Gamma(s)$ is the gamma function, and the branch of the logarithm is determined by continuous variation up the vertical line, starting from $\theta(0)=0$. The theta function takes real arguments $t \in \mathbb R$, and  is a real-valued function (as can be seen from the definition). It can also be shown that near $t_{min}\approx 7$ it has a minimum value of $\theta(t_{min}) \approx -3.5 $, while for $t> t_{min}$, $\theta(t)$ is strictly increasing. And by applying Stirling's formula for the gamma function, we can derive an asymptotic expansion for theta which, although does not converge, its first few terms give a good approximation when $t \gg 1$
\begin{equation} \label{eq:asympt}
\theta(t) = \frac{t}{2}\log\frac{t}{2\pi} -\frac{t}{2}-\frac{\pi}{8} + \sum_{k=1}^n \frac{B_{2k} (1-2^{1-2k})}{4k(2k-1)} \cdot \frac{1}{t^{2k-1}} + \mathcal O\left( \frac{1}{t^{2n+1}} \right) 
\end{equation}
(here $B_{2k}$ are the Bernoulli numbers).

The other function is called the Hardy Z function, and is given by
\begin{equation}
Z(t) := e^{i\theta(t)} \zeta\left( \frac{1}{2}+i t \right)                  \label{eq:hardyz}
\end{equation}
Similar to theta, the Z function also takes real arguments $t\in \mathbb R$, and it is a real-valued function (as a consequence of the functional equation for the zeta function). We remark that  the zeros of $\zeta(\frac{1}{2}+it)$ coincide with the zeros of $Z(t)$. But because the Z function is real rather than complex, these zeros can be found simply by studying its sign changes. And in turn, the changes in the sign of $Z(t)$ can be determined from evaluating $\zeta(\frac{1}{2}+it)$ at certain points, named Gram points.

\begin{definition}
For any integer $M \geq -1$, we define the $M$-th Gram point $g_M$ as the unique solution in the range $t > 7$ of the equation
\begin{equation} \label{eq:gram} 
\theta(g_M) = M\pi 
\end{equation}
and we call a Gram interval any interval between two consecutive Gram points $[ g_M, \, g_{M+1} )$.
\end{definition}

The definition of the Z function \eqref{eq:hardyz}, together with Euler's identity, imply that at every Gram point $g_M$ we have
\begin{equation} 
\zeta\left(\frac{1}{2}+ig_M\right) = (-1)^M Z(g_M)  \label{eq:zeta}
\end{equation}
Keeping in mind that  $Z(t)$ is a real-valued function, we obtain an alternative definition of Gram points, namely as points on the critical line at which the Riemann zeta function $\zeta\left( \frac{1}{2}+i g_M \right) $   takes real (non-zero) values. 

In particular, if $ \zeta\left( \frac{1}{2}+it \right)  $ has the same sign at two successive Gram points $t=g_M $ and $t=g_{M+1}$, then according to \eqref{eq:zeta} $Z(t)$ must have opposite signs at these points. This means that $Z(t)$ has at least a root between $g_M$ and $g_{M+1}$, which is equivalent to $\zeta(\frac{1}{2}+it)$ having at least one zero inside the Gram interval $[g_M, g_{M+1})$. 

This technique was initially used by Danish mathematician J\o rgen P. Gram \cite{gram} in 1903 to find the first 15 zeros of $\zeta(\frac{1}{2}+it)$  in the range $0< t < 66$. He noticed that $\zeta\left(\frac{1}{2}+ig_M\right) > 0$ for all $-1\leq M \leq 14$ and that each of these Gram intervals contained exactly one zero of the zeta function or, in other words, that the Gram points alternated with the zeta zeros. Gram believed that this pattern would continue beyond the first 15 intervals, but also that it would not necessarily hold true all the time. When it does hold, this phenomenon is named Gram's Law.

\begin{definition}
Given two consecutive Gram points $g_M$ and $g_{M+1}$, we say that Gram's Law holds true for $[ g_M, \, g_{M+1} )$ if this Gram interval contains exactly one zero of $\zeta(\frac{1}{2}+it)$.
\end{definition}

The original definition was proposed by  J. I. Hutchison \cite{hutchinson} in 1925, and was given in terms of the zeros of the Z function  \\

\begin{quote}
``Gram calculated the first fifteen roots [of $Z(t)$] and called attention to the fact that the [roots] and the [Gram points] separate each other. I will refer to this property of the roots as \emph{Gram's Law}. Gram expressed the belief that this law is not a general one." \\
\end{quote} 

Hutchison also extended Gram's computations to the first 138 zeros of $\zeta(\frac{1}{2}+it)$, and discovered the first instances where Gram's Law fails: the interval $[g_{125} , g_{126})$ doesn't contain any zeros, while the next one $[g_{126} , g_{127})$ has two. Subsequently, in 1935 it was proved by E. C. Titchmarsh \cite{titchmarsh2} that it fails infinitely many times. 

There is another, less restrictive version of Gram's Law, called the Weak Gram's Law, which states that for $M\in\mathbb N$
\[ (-1)^M Z(g_M)>0 \quad \text{ and } \quad (-1)^{M+1}Z(g_{M+1})>0 \]
This is equivalent to claiming that the Gram interval $[g_M, g_{M+1})$ contains an odd number of simple zeros (or a zero with odd multiplicity). Although this version also has exceptions (as can be seen from Hutchison's results) Titchmarsh \cite{titchmarsh1} managed to prove that the Weak Gram Law is true infinitely many times. And more recently, in 2009, T. S. Trudgian  \cite{trudgian1}, \cite{trudgian2} has shown that for sufficiently large $T$, there exists a positive proportion of Gram intervals between $T$ and $2T$ that contain at least one zero of $\zeta(\frac{1}{2}+it)$; in particular, this implies that the Weak Gram Law is true for a positive proportion of the time. 

In the case of the original Gram's Law, it has not yet been proven whether it is also true infinitely many times, much less for a positive proportion of the time. Despite this uncertainty, extensive numerical computations at very high regions up the critical line suggest that Gram's Law does hold  for a large proportion of these intervals: in a series of four papers published during the 1980's, R. P. Brent, J. van de Lune and others \cite{brent1}, \cite{brent2}, \cite{lune1}, \cite{lune2} have analyzed the first 1.5 billion Gram intervals (up to height $t=545,439,823.215$) and reported that approximately $72.61\%$ of them obey Gram's Law. A summary of their results can be seen in Table \ref{tab:brentlune} below.

\begin{definition}
For any $k\in \mathbb N \cup \{0\}$ and $0\leq L<M$, we define $G_{L, M}(k) \in [0,1]$ to be the proportion of Gram intervals between $g_L$ and $g_M$ that contain exactly $k$ zeros (in particular, $G_{L, M}(1)$ represents the proportion of intervals that obey Gram's Law, while $(M-L)\cdot G_{L,M}(k)$ will be the number of Gram intervals with exactly $k$ zeros).
\end{definition}

\begin{remark}
In everything that follows from here on, the first zeta zero $\gamma_1 = 14.1347\ldots$ and the corresponding Gram interval $[g_{-1}, g_0)$ are going to be excluded.
\end{remark}

\begin{table}[h]
  \centering
\begin{tabular}{| r | r | r | r | r | r |}
\hline
 $M$ &  $M \cdot G_{0,M}(0)$  &  $M \cdot G_{0,M}(1)$  &  $M \cdot G_{0,M}(2)$  &  $M \cdot G_{0,M}(3)$  &  $M \cdot G_{0,M}(4)$  \\
\hline   
100                    &                       &                               100  &                                        &                                        &                                       \\
1,000                 &                  42 &                               916  &                                   42 &                                        &                                       \\
10,000               &                808 &                             8,390 &                                 796 &                                    6  &                                       \\
100,000             &           10,330 &                           79,427 &                           10,157 &                                  86  &                                       \\
1,000,000          &         116,055 &                        769,179  &                        113,477  &                            1,289   &                                       \\
10,000,000        &      1,253,556 &                      7,507,820 &                     1,223,692  &                           14,932  &                                       \\
100,000,000      &    13,197,331 &                    73,771,910 &                    12,864,188 &                          166,570 & 1                                    \\
1,000,000,003   &  137,078,283 &                  727,627,708  &                133,509,764  &                      1,784,225  &  23                                 \\
\hline
\end{tabular}
\medskip
\caption{Summary of results by R. Brent, J van de Lune et al.} \label{tab:brentlune}
\end{table}

Van de Lune et al. have concluded the last paper in their series with the following remarks: \\

\begin{quote}
``Our statistical material suggests that the zeros of $Z(t)$ are distributed among the Gram intervals according to some hitherto unknown probabilistic law. (\dots) It would be interesting to have a probabilistic model which could explain or at least support this phenomenon.'' \\
\end{quote} 

The main purpose of the current paper is to use a conjecture from Random Matrix Theory (RMT) to develop such a model, that describes the asymptotic limit of $G_{0,M}(k)$ for large $M$, as well as its rate of convergence. 

\subsection{RMT and Fujii's conjecture}

Following work of F. Dyson \cite{dyson} and H. Montgomery \cite{montgomery}, a conjecture was established --- backed up by theoretical, heuristic, and numerical evidence --- that generalizes their results and is equivalent to the following statement

\begin{conjecture}[Montgomery, Dyson et. al.] 
The zeros of the Riemann zeta function at height $T$ on the critical line are statistically distributed like the eigenvalues of a $N\times N$ random  unitary matrix around the unit circle, where the height $T$ and the matrix size $N$ are related by
\begin{equation} \label{eq:nt} 
N\approx \log\frac{T}{2\pi} 
\end{equation}
\end{conjecture}

In order for this kind of comparison to make sense, the zeta zeros must be normalized to have unit average spacing, while the eigenangles are rescaled to have unit mean density. 


As usual, we  denote by U($N$) the group of all $N\times N$ unitary matrices, and for any matrix $A\in $ U($N$) we will denote its eigenvalues by $e^{i\theta_1}, \dots, e^{i\theta_N}$, where $\theta_1, \ldots, \theta_N \in [-\pi, \pi)$. A random element of this group means being chosen according to Haar measure, which is the only probability measure on U($N$) that is invariant  under unitary transformations. If we define the U($N$) Dyson product to be

\[ \mathcal D_{\U(N)}(\theta_1, \ldots, \theta_N) := \prod_{1\leq j<k \leq N} |e^{i\theta_j} - e^{i\theta_k}|^2 \]
then Weyl \cite{weyl} showed that Haar measure on the U($N$) group leads to the following probability density function for the eigenangles

\[ \mathcal P_{\U(N)}( \theta_1, \dots, \theta_N ) := \frac{1}{N!(2\pi)^N} \mathcal D_{\U(N)}(\theta_1, \ldots, \theta_N) \]
(in other words, $\mathcal P_{\U(N)}( \theta_1, \dots, \theta_N ) \; d\theta_1 \ldots d\theta_N$ represents the probability of a random U($N$) matrix having  eigenangles in $[ \theta_1, \> \theta_1+d\theta_1]$, $[ \theta_2, \> \theta_2+d\theta_2]$, and so on).

\begin{definition} 
Let $J\subset [ -\pi,\pi ) $ be an arbitrary fixed interval on the unit circle, of length $ \frac{2\pi}{N}$. The probability that $J$ contains exactly $k$ unscaled eigenvalues of a random $\U(N)$ matrix is given by

\begin{equation} \label{eq:eun}
E_{\U(N)}(k,J) := \binom{N}{k} \int_{J^k} \int_{([-\pi,\pi)\smallsetminus J)^{N-k}} \mathcal P_{\U(N)}(\theta_1, \ldots, \theta_N) \; d\theta_1\ldots d\theta_N
\end{equation}
\end{definition}

 In Section 3 we will present a more efficient formula for computing this probability. Note that by rotation invariance of Haar measure, $E_{\U(N)}(k,J)$ is insensitive to the actual starting  position of $J$, only to its length.  Also, it can be easily shown, by a change of variables, that the probability of finding $k$ unscaled eigenvalues in an interval of length $\frac{2\pi}{N}$ is equal to the probability of finding $k$ rescaled eigenvalues in an interval of length 1.

Because the egienvalues of a random unitary matrix provide a good statistical model for the zeros of the zeta function, it is natural to ask if there could also exist a RMT model for Gram's Law. The first such model was proposed by A. Fujii  \cite{fujii} in 1987, who made a conjecture that is equivalent to the following statement:

\begin{conjecture}[Fujii]
For any $k\in \mathbb N\cup\{0\}$ 

\[ \lim_{N\to\infty} E_{U(N)}(k, J)  = \lim_{M\to\infty}  G_{0,M}(k) \]
\end{conjecture}
\quad \\
If $k=0,1,2$ the values of the limits on the LHS are known to be approximately \cite{odlyzko2}
\begin{align*}
 & \lim_{N\to\infty }E_{U(N)}(0, J) \approx 0.17022 \\
 & \lim_{N\to\infty }E_{U(N)}(1, J) \approx 0.66143 \\
 & \lim_{N\to\infty }E_{U(N)}(2, J) \approx 0.16649 
\end{align*}
One way of verifying whether the two limits in Fujii's conjecture do indeed coincide is to compare their rates of convergence and see how similar they are. Using formula \eqref{eq:eun} we can compute the values of $E_{U(N)}(k,J)$ for small $N$, which are given in Table \ref{tab:E_U(N)(k,J)}. We remark that on each column, $E_{U(N)}(k,J)$ converges very fast to its corresponding limit, and that on every row, they add up to almost 100\%.


\begin{table}[h]
  \centering
\begin{tabular}{| r | c | c | c |}
\hline
 $N$ &  $E_{U(N)}(0, J)$  &  $E_{U(N)}(1, J)$  &  $E_{U(N)}(2, J)$  \\
\hline
 2   &  0.148679  &  0.702642  & 0.148679 \\
 3   &  0.161362  &  0.678268  & 0.159378 \\
 4   &  0.165362  &  0.670641  & 0.162630 \\
 5   &  0.167146  &  0.667251  & 0.164060 \\
 6   &  0.168098  &  0.665445  & 0.164817 \\
 7   &  0.168666  &  0.664367  & 0.165268 \\
 8   &  0.169032  &  0.663673  & 0.165558 \\
 9   &  0.169283  &  0.663199  & 0.165755 \\
10  &  0.169461  &  0.662860  & 0.165896 \\
11  &  0.169593  &  0.662611  & 0.166000 \\
12  &  0.169693  &  0.662421  & 0.166079 \\
13  &  0.169771  &  0.662274  & 0.166140 \\
14  &  0.169833  &  0.662157  & 0.166188 \\
15  &  0.169882  &  0.662063  & 0.166227 \\
16  &  0.169923  &  0.661986  & 0.166259 \\
17  &  0.169957  & 0.661922   & 0.166286 \\
18  &  0.169985  & 0.661869   & 0.166308 \\
19  &  0.170009  & 0.661824   & 0.166327 \\
20  &  0.170029  & 0.661785   & 0.166343 \\
21  &  0.170047  & 0.661752   & 0.166357 \\
\hline
\end{tabular}
\medskip
 \caption{$E_{U(N)}(k, J)$ for $k=0,1,2$ and $N=2,\ldots, 21$} \label{tab:E_U(N)(k,J)}
\end{table}

However, the values of $G_{0,M}(k)$ that can be deduced from Table \ref{tab:brentlune} are not directly comparable with the entries from Table \ref{tab:E_U(N)(k,J)}, and need to be recomputed for different indexes $M$ in the following way: we know that the matrix size $N$ is related to the height up the critical line $T$  according to formula \eqref{eq:nt}, which is equivalent to 
\[ T \approx 2\pi e^N \]
And in our case, the height is given by the Gram points that we are interested in ($T=g_M$), so we can define $M_N$ to be the index of the Gram point $g_{M_N}$ that lies at the height on the critical line that corresponds to unitary matrices of size $N \times N$. 

From the definition of the Gram points \eqref{eq:gram} we have that

\[ M_N = \frac{1}{\pi} \theta(g_{M_N}) \]
and from the asymptotic formula for the theta function \eqref{eq:asympt} we get

\[ \frac{1}{\pi} \theta(T) \approx \frac{T}{2\pi}\log\frac{T}{2\pi} -\frac{T}{2\pi} \]

Combining the previous three equations, we can derive an approximate formula for the index $M_N$ in terms of the matrix size

\[ M_N  \approx e^N (N-1) \]
The Gram points with these indexes split the critical line into increasingly large segments of type $[g_{M_N}, g_{M_{N+1}})$, and for each segment we can recompute the proportion of Gram intervals that contain exactly $k$ zeros, $G_{M_N , M_{N+1}}(k)$; these are presented in Table \ref{tab:GramStatsLowHeight} and represent the values that can be compared with the corresponding  results from Table \ref{tab:E_U(N)(k,J)}.

\begin{table}[h]
  \centering
\begin{tabular}{ | r | r | c | c | c |}
\hline
$N$ & $M_N$  &  $G_{M_N , M_{N+1}}(0)$  &  $G_{M_N , M_{N+1}}(1)$  &  $G_{M_N , M_{N+1}}(2)$  \\
\hline
2   &                          7 &                 & 1.000000 &                   \\
3   &                        40 & 0.016129 & 0.967741 &  0.016129 \\
4   &                      164 & 0.039351 & 0.921296 & 0.039351  \\
5   &                      594 & 0.069669 & 0.860661 & 0.069669  \\
6   &                   2,017 & 0.083059 & 0.834538 & 0.081744  \\
7   &                   6,580 & 0.095051 & 0.810527 & 0.093791  \\
8   &                 20,867 & 0.105168 & 0.790572 & 0.103348  \\
9   &                 64,825 & 0.111233 & 0.778687 & 0.108924  \\
10 &               198,238 & 0.116361 & 0.768576 & 0.113764  \\
11 &               598,741 & 0.121410 & 0.758585 & 0.118597  \\
12 &            1,790,303 & 0.125309 & 0.750841 & 0.122389  \\
13 &            5,308,961 & 0.128694 & 0.744212 & 0.125490  \\
14 &          15,633,856 & 0.131542 & 0.738581 & 0.128210  \\
15 &          45,766,243 & 0.134146 & 0.733422 & 0.130716 \\
16 &        133,291,658 & 0.136422 & 0.728930 & 0.132871 \\
17 &        386,479,244 & 0.138428 & 0.724956 & 0.134802 \\
18 &     1,116,219,475 & 0.140223 & 0.721401 & 0.136526 \\
19 &     3,212,681,417 & 0.141825 & 0.718223 & 0.138077 \\
20 &     9,218,138,713 & 0.143277 & 0.715342 & 0.139481 \\
21 &   26,376,314,690 & 0.144590 & 0.712736 & 0.140756 \\
22 &   75,283,169,769 &     \dots    &    \dots     & \dots        \\
\hline
\end{tabular}
\medskip
\caption{ $G_{M_N , M_{N+1}}(k)$ for $k=0,1,2$ and $N=2,\ldots, 21$ }  \label{tab:GramStatsLowHeight}
\end{table}

Now, according to Conjecture 1, for each finite $N$, $E_{\U(N)}(k, J)$ should provide a good approximation to $G_{M_N , M_{N+1}}(k)$, but we can see that it is not the case. This does not necessarily imply that their asymptotic limits don't coincide, but what is clear from this data is that $E_{\U(N)}(k, J)$ and  $G_{M_N , M_{N+1}}(k)$ have very different rates of convergence.

We claim that this apparent contradiction with Conjecture 1 originates from the use of an incorrect RMT analogy for Gram points and intervals. Recall that $J$ was chosen to be an arbitrary fixed interval on the unit circle and as a consequence, there is nothing inherently special about its endpoints. However, these endpoints should represent the RMT analogues of Gram points and, as we have mentioned, the Gram points are special, in the sense that they are points on the critical line at which the zeta function takes real (non-zero) values.

In the following section, we will consider and analyze an alternative RMT model for Gram's Law, in which the analogues Gram points are not fixed on the unit circle, but instead depend on the corresponding unitary matrix and are related to its characteristic polynomial in the same manner in which the actual Gram points relate to the Riemann zeta function. Towards this purpose, we first  introduce the corresponding notions of probability density function and Dyson product for random SU($N$) matrices, which will be required later.


\subsection{Random special unitary matrices}

A SU($N$) matrix is a unitary matrix with determinant equal to 1. It also has a Haar measure which effectively comes from the Haar measure for unitary matrices, but with one eigenangles forced to equal the value that makes the sum of all $N$ eigenangles congruent to $0 \pmod{2\pi}$ since that would make the determinant equal to 1 \cite{hiai}.

That is, the probability density function for the $N$ eigenangles of a Haar distributed $\SU(N)$ matrix is 
\[ \mathcal P_{SU(N)}(\theta_1, \ldots, \theta_N) := \frac{1}{N!(2\pi)^{N-1}} \mathcal D_{U(N)}(\theta_1, \ldots, \theta_N) \cdot \delta(\theta_1+\ldots+\theta_N \text{ mod } 2\pi) \]
where $\delta(x)$ represents the Dirac delta function. If we integrate it over one of the variables, we have
\[ \int_{[-\pi, \pi)} \mathcal P_{SU(N)}(\theta_1, \ldots, \theta_N) \, d\theta_N =  \frac{1}{N!(2\pi)^{N-1}} \mathcal D_{SU(N)}(\theta_1, \ldots, \theta_{N-1}) \]
where $\mathcal D_{SU(N)}(\theta_1, \ldots, \theta_{N-1})$ denotes the SU($N$) Dyson product, and  is given by
\begin{align*}
    \mathcal D_{SU(N)}(\theta_1, \ldots, \theta_{N-1})  &:= \int_{[-\pi, \pi)}\mathcal  D_{U(N)}(\theta_1, \ldots, \theta_N) \cdot \delta(\theta_1+\ldots+\theta_N \text{ mod } 2\pi ) \, d\theta_N \\
    &= \mathcal D_{U(N)}(\theta_1, \ldots, \theta_{N-1} , -\theta_1 - \ldots - \theta_{N-1})  \\
    &= \prod_{1\leq  j<k \leq N-1} |e^{i\theta_j}-e^{i\theta_k}|^2  \prod_{1\leq k \leq N-1} |e^{i\theta_k}-e^{-i(\theta_1+\ldots+\theta_{N-1})}|^2
\end{align*}

\section{Gram's Law for random matrices}

\subsection{$\U(N)$ Gram points and intervals}

In order to motivate our RMT equivalent of Gram points, we recall from the previous section the way in which the zeta function is related to the Z function
\[ \zeta\left( \frac{1}{2}+it \right) = Z(t) e^{-i\theta(t)} = Z(t)\cos\theta(t) - iZ(t)\sin\theta(t)  \]
If we want to find the points on the critical line at which the zeta function is real, we have to impose the condition that its imaginary part should be zero, from which we get
\[ \zeta\left( \frac{1}{2}+it \right)\in\mathbb R \quad \Leftrightarrow \quad \Im \zeta\left( \frac{1}{2}+it \right)=0 \quad \Leftrightarrow \quad Z(t)\sin\theta(t)=0 \]
The last condition is equivalent to the following two possibilities
\begin{itemize}
\item $ Z(t)=0$, which also gives all the zeros of $\zeta ( \frac{1}{2}+it )$;
\item $ \sin\theta(t)=0 \, \Leftrightarrow \, \theta(t)=M\pi$ for $M\in \mathbb Z$, from which we get the Gram points $g_M$.
\end{itemize}

Because the eigenvalues of a random U($N$) matrix are the RMT analogues of the zeta zeros, and the unitary circle represents the analogue of the critical line, J. Keating and N. Snaith \cite{keatingsnaith} introduced the characteristic polynomial of a unitary matrix as a RMT model for the Riemann zeta function; if $A\in$ U($N$), this is defined as
\[ \Lambda_A(\theta) :=  \det(I_N - Ae^{-i\theta} ) \]

It can be re-expressed in terms of the matrix eigenvalues $e^{i\theta_1}, \dots, e^{i\theta_N}$ as
\begin{align*}
\Lambda_A(\theta) &=  \prod_{j=1}^N(1-e^{i(\theta_j-\theta)}) \\
&= \prod_{j=1}^N \exp{\frac{i(\theta_j-\theta)}{2} } \left[ \exp{\left(-\frac{i(\theta_j-\theta)}{2} \right) } - \exp{\left(\frac{i(\theta_j-\theta)}{2}\right)} \right]   \\
& = (-2i)^N \prod_{j=1}^N \left[ \exp{\left(\frac{i(\theta_j-\theta)}{2} \right)} \sin \left( \frac{\theta_j-\theta}{2}\right) \right]  \\
& = (-2)^N \exp{ \left( i\frac{N\pi}{2} \right) } \exp{\left( i\sum_{j=1}^N \frac{\theta_j-\theta}{2}\right)}  \prod_{j=1}^N \sin \left( \frac{\theta_j-\theta}{2} \right)  \\
&  = (-2)^N \exp{\left[i \left(\frac{N\pi}{2} + \frac{\theta_1+\ldots+\theta_N}{2} - \frac{N\theta}{2} \right)\right] } \prod_{j=1}^N \sin \left( \frac{\theta_j-\theta}{2}\right)
\end{align*}
We continue the above analogy by searching for the points $\theta\in [-\pi, \pi)$ on the unit circle at which the characteristic polynomial is real
\[  \Lambda_A(\theta)\in\mathbb R \quad \Leftrightarrow \quad \Im \Lambda_A(\theta)= 0 \quad \Leftrightarrow \quad \sin\left( \frac{N\pi}{2} + \frac{\theta_1+\ldots+\theta_N}{2} -\frac{N\theta}{2} \right) \prod_{j=1}^N \sin\frac{\theta_j-\theta}{2} = 0               \]
As before, this leads to two possible cases

\begin{itemize}
\item $ \displaystyle \prod_{j=1}^N \sin\frac{\theta_j-\theta}{2}=0$
\item $ \displaystyle \sin\left( \frac{N\pi}{2} + \frac{\theta_1+\ldots+\theta_N}{2} -\frac{N\theta}{2} \right) = 0 \; \Leftrightarrow \; \frac{N\pi}{2} + \frac{\theta_1+\ldots+\theta_N}{2} -\frac{N\theta}{2} = m\pi$, for some $m\in\mathbb Z $.
\end{itemize}
From the first condition  we recover the $N$ eigenangles  $\theta \in \{\theta_1, \dots , \theta_N\}$ (which are the U($N$) analogues of the zeta zeros). From the second condition, we obtain another set of points, given by
\[ \theta \in \left\{ \frac{\theta_1+\ldots+\theta_N}{N} + \pi -\frac{2m\pi }{N} , \quad m \in \mathbb Z \right\} \]
We note that only $N$ elements of this set are distinct modulo $2\pi$, and because they represent the points on the unit circle at which the characteristic polynomial of a $\U(N)$ matrix is real (but not necessarily zero), we will consider them to be the analogous $\U(N)$ Gram points. 

\begin{definition} If $A$ is a $\U(N)$ matrix with eigenvalues $e^{i\theta_1}, \dots, e^{i\theta_N}$, we define the corresponding $\U(N)$ Gram points as
\[ \psi^{U(N)}_m := \frac{\theta_1+\ldots+\theta_N}{N} - \pi + \frac{2m\pi }{N} , \quad   m = 0,1,\ldots, N-1    \]
We also define a $\U(N)$ Gram interval as any interval on the unit circle between two consecutive $\U(N)$ Gram points.
\end{definition}

We remark that the $\U(N)$ Gram points are placed along the unit circle at equal distance from each other in steps of $\frac{2\pi}{N}$, rather than being distributed arbitrarily. Furthermore, they are not fixed on the unit circle, and are not the same for all $A\in$ U($N$) matrices, but instead depend on $\arg(\det A) = \theta_1 + \ldots + \theta_N \pmod{2\pi}$. With these definitions in mind, we can analyze what is the probability of having exactly $k$ eigenvalues of a random $\U(N)$ matrix inside one of these $\U(N)$ Gram intervals, in order to understand if and how it differs from $E_{U(N)}(k, J)$. 

We begin by studying the simplest case, that of $N=2$, which can be solved using just elementary logic, without any computations. According to the formula given above, if $A$ is a unitary matrix of size $2\times 2$, then its eiegenangles $\theta_1, \theta_2$ are related to its U(2) Gram points $\psi_1, \psi_2$ by
\[ \psi_1 = \frac{\theta_1 + \theta_2}{2}-\pi \quad \text{and} \quad \psi_2=\frac{\theta_1+\theta_2}{2} \]
Now, since $\psi_2$ is the arithmetic average of $\theta_1$ and $\theta_2$, this means that it will always be located between them on the unit circle (regardless of where they are). On the other hand, $\psi_1$ is diametrically opposed to $\psi_2$, so it will also lie between $\theta_1$ and $\theta_2$,  but on  the other side of the circle. This is equivalent to having the two $\theta_j$'s positioned between the two $\psi_j$'s, each one on a different arc. In particular, this implies that the probability of finding exactly $k=1$ eigenvalue of a random U(2) matrix inside a U(2) Gram interval will always be 100\%; it also means that the probability is zero for having an empty  U(2) Gram interval ($k=0$) or of having both eigenvalues in the same interval ($k=2$). These results are not only very different from the values on row $N=2$ of Table \ref{tab:E_U(N)(k,J)} but, more importantly, are in perfect agreement with the entries on row $N=2$ of Table \ref{tab:GramStatsLowHeight} 

This hints at the more general fact that the probability of finding exactly $k$ eigenvalues of a random U($N$) matrix in a U($N$) Gram interval gives a much better model for  $G_{M_N , M_{N+1}}(k)$ than $E_{\U(N)}(k, J)$. However, it becomes increasingly difficult to compute this quantity in a direct way for $N\geq 3$ (the problem comes from the fact that this probability is essentially an integral over the eigenangles $\theta_1, \dots, \theta_N$ and each $\U(N)$ Gram point depends on all of them). In order to overcome this difficulty, we will relate this quantity to the corresponding probability for a particular kind of $\U(N)$ matrices, namely the special unitary matrices, and then focus on computing that probability.\\

\subsection{$\SU(N)$ Gram points and intervals}

If $A$ is a $\SU(N)$ matrix then, by definition, $\arg(\det A) = \theta_1 + \ldots + \theta_N = 0 \pmod {2\pi}$, and the above U($N$) Gram points are reduced to, what we will call, the $\SU(N)$ Gram points.

\begin{definition} We define the $\SU(N)$ Gram points as
\[ \psi^{SU(N)}_m := - \pi + \frac{2m\pi }{N} , \quad   m = 0,1,\ldots, N-1    \]
We also define an $\SU(N)$ Gram interval as any interval along the unit circle between two consecutive $\SU(N)$ Gram points.
\end{definition}

Similar to the $\U(N)$ case, these represent the points on the unit circle at which the characteristic polynomial of a $\SU(N)$ matrix is real (but not necessarily zero), and they are distributed equidistant in steps of $\frac{2\pi}{N}$. However, unlike the $\U(N)$ case, the SU($N$) Gram points do not depend in any way on the eigenangles, which implies that they are the same for all SU($N$) matrices, and are also fixed on the unit circle. As we will  later see, this makes it easier to compute the probability of having exactly $k$ eigenvalues of a random $\SU(N)$ matrix inside a $\SU(N)$ Gram interval. For now, we will prove the following result, which relates this quantity with the corresponding probability from the previous subsection:

\begin{lema} For any $k=0, \dots, N$, we have that
\begin{align*}
& \text{Pr }[\text{exactly } k \text{ eigenvalues of a } U(N) \text{ matrix lie in a } U(N) \text{ Gram interval}] = \\
& = \text{Pr }[\text{exactly } k \text{ eigenvalues of a } SU(N) \text{ matrix lie in a } SU(N) \text{ Gram interval}] 
\end{align*}
where the first probability is over Haar measure for $\U(N)$ and the second probability is over Haar measure for $\SU(N)$.
\end{lema}

\begin{proof}
Let $e^{i\theta_1},\dots,e^{i\theta_N}$ be the eigenvalues of a U$(N)$ matrix. For simplicity, we will use
\[ \mathcal I = \left[ \frac{\theta_1+\ldots +\theta_N}{N} - \pi, \quad \frac{\theta_1+\ldots +\theta_N}{N} - \pi+\frac{2\pi}{N} \right) \pmod{2\pi} \]
as a generic $\U(N)$ Gram interval and denote its complement by
 \[ [-\pi, \pi)\smallsetminus \mathcal I = \left[ \frac{\theta_1+\ldots +\theta_N}{N} - \pi + \frac{2\pi}{N}, \quad \frac{\theta_1 + \ldots + \theta_N}{N} + \pi \right) \pmod{2\pi} \]

Because the $\U(N)$ probability density is a symmetric function in all eigenangles, it can be shown that the probability of having $k$ eigenvalues of a $\U(N)$ matrix in a $\U(N)$ Gram interval is the same, for any $\U(N)$ Gram interval and for any $k$ eigenvalues. Starting with this fact, we have that
\begin{align}
  & \text{Pr }[\text{exactly } k \text{ eigenvalues of a U}(N) \text{ matrix lie in a U}(N) \text{ Gram interval}] \notag\\
  & = \binom{N}{k} \text{Pr }[\theta_1, \dots , \theta_k \in \mathcal I \> \text{ and } \> \theta_{k+1}, \dots , \theta_N\in [-\pi, \pi)\smallsetminus \mathcal I ] \notag \\
  & = \binom{N}{k} \int \dots \int_{\mathcal R} \mathcal P_{U(N)} (\theta_1, \dots , \theta_N) \; d\theta_1 \dots d\theta_N   \notag \\ 
  & = \binom{N}{k} \frac{1}{N!(2\pi)^N}\int\dots\int_{\mathcal R}   \mathcal D_{U(N)}(\theta_1, \dots, \theta_N) \; d\theta_1\dots d\theta_N \label{eq:Prob_k_evals_UN}
\end{align}
where the $N$-dimensional integral is over a region $\mathcal R$ described by the restrictions
\[ 
\mathcal R : 
\begin{cases}
	\theta_1, \dots , \theta_k \in \mathcal I \\
	\theta_{k+1}, \dots , \theta_N\in [-\pi, \pi) \smallsetminus \mathcal I
\end{cases}
\]
which can be written out explicitly as
\[
\mathcal R : 
\begin{cases}
	\displaystyle \frac{\theta_1 + \ldots + \theta_N}{N} - \pi \leq \theta_n < \frac{\theta_1 + \ldots + \theta_N}{N} - \pi + \frac{2\pi}{N}    \quad (n = 1,\dots, k)  \medskip  \\
	\displaystyle \frac{\theta_1 + \ldots + \theta_N}{N} - \pi + \frac{2\pi}{N} \leq \theta_n < \frac{\theta_1 + \ldots + \theta_N}{N} + \pi   \quad (n = k+1,\dots, N)
\end{cases}
\]
\[ 
\Leftrightarrow \quad \mathcal R : 
\begin{cases}
	\displaystyle -\pi \leq \theta_n - \frac{\theta_1 + \ldots + \theta_N}{N} < -\pi + \frac{2\pi}{N}   \quad  (n = 1,\dots, k)  \medskip    \\
	\displaystyle -\pi+\frac{2\pi}{N} \leq \theta_n - \frac{\theta_1 + \ldots + \theta_N}{N} < \pi       \quad (n = k+1,\dots, N)
\end{cases}
\]
\quad \\
We now perform the following change of variables:
\begin{align*}
	\lambda_n & = \theta_n - \frac{\theta_1 + \ldots + \theta_N}{N}   \quad (n = 1,\dots, N-1)  \medskip  \\
	\lambda_N & = N \theta_N
\end{align*}
The determinant of the Jacobian matrix is 1. Note that modulo $2\pi$ the restriction on $\theta_N$ is lost when it comes to considering $\lambda_N$. Furthermore, with this change of variables, we have that
\[ \theta_N -  \frac{\theta_1 + \ldots + \theta_N}{N} = - \lambda_1 - \ldots - \lambda_{N-1} \]
and the previous region of integration $\mathcal R$ is now described by the conditions
\[ \mathcal R' :
\begin{cases}
	\displaystyle -\pi \leq \lambda_n < -\pi+\frac{2\pi}{N}       \quad (n = 1,\dots, k)  \medskip \\
	\displaystyle -\pi+\frac{2\pi}{N} \leq \lambda_n <\pi         \quad (n = k+1,\dots, N-1)  \medskip \\
	\displaystyle -\pi+\frac{2\pi}{N} \leq -\lambda_1- \ldots - \lambda_{N-1} < \pi
\end{cases}
\]
If we denote a generic $\SU(N)$ Gram interval and its complement by
\[ \mathcal J = \left[ -\pi, \> -\pi+\frac{2\pi}{N} \right) \quad \text{and} \quad [-\pi, \pi) \smallsetminus \mathcal J = \left[ -\pi+\frac{2\pi}{N} , \> \pi \right) \]
then $\mathcal R'$ becomes
\[ \mathcal R' :
\begin{cases}
	\lambda_1, \dots , \lambda_n \in \mathcal J \\
	\lambda_{n+1}, \dots , \lambda_{N-1} \in [-\pi, \pi)\smallsetminus \mathcal J \\
	-\lambda_1- \ldots - \lambda_{N-1}  \in [-\pi, \pi)\smallsetminus \mathcal J
\end{cases}
\]
Since there is no restriction imposed on $\lambda_N$, it can be taken $\lambda_N \in [-\pi, \pi)$

The old variables $\theta_n$ can be expressed in terms of the new variables $\lambda_n$ as:
\begin{align*}
	\theta_n & = \lambda_n + (\lambda_1 + \ldots + \lambda_{N-1}) + \frac{\lambda_N}{N}  \quad (n = 1,\dots, N-1)   \medskip \\
	\theta_N & =\frac{\lambda_N}{N}
\end{align*}
We note that
\[ \theta_m -\theta_n = \lambda_m -\lambda_n \quad \text{for} \quad m,n = 1,\dots, N-1 \]
and
\[ \theta_n - \theta_N= \lambda_n + (\lambda_1 + \ldots + \lambda_{N-1})  \quad \text{for} \quad  n = 1,\dots, N-1 \]
which implies that
\begin{align*}
\mathcal D_{U(N)}(\theta_1, \dots, \theta_N) &= \prod_{1\leq m<n \leq N} |e^{i\theta_m}-e^{i\theta_n}|^2 \\
&= 2^{N(N-1)} \prod_{1\leq m<n \leq N} \left( \sin\frac{\theta_m -\theta_n}{2} \right)^2  \\
& = 2^{N(N-1)}  \prod_{1\leq m<n \leq N-1} \left( \sin\frac{\theta_m - \theta_n}{2} \right)^2 \prod_{1\leq n \leq N-1} \left( \sin\frac{\theta_n - \theta_N}{2} \right)^2  \\
& = 2^{N(N-1)} \prod_{1\leq m<n \leq N-1} \left( \sin\frac{\lambda_m - \lambda_n}{2} \right)^2 \prod_{1\leq  n \leq N-1} \left( \sin\frac{\lambda_n + (\lambda_1 + \ldots + \lambda_{N-1})}{2} \right)^2        \\
& =  \prod_{1\leq m<n \leq N-1} |e^{i\lambda_m}-e^{i\lambda_n}|^2  \prod_{1\leq n \leq N-1} |e^{i\lambda_n}-e^{-i(\lambda_1 + \ldots + \lambda_{N-1})}|^2 \\
&= \mathcal D_{SU(N)}(\lambda_1, \dots , \lambda_{N-1})
\end{align*}

Putting everything together, we obtain that the initial integral \eqref{eq:Prob_k_evals_UN} can be expressed in the new system of variables as
\begin{align*}
& \binom{N}{k} \frac{1}{N!(2\pi)^N}\int\dots\int_{\mathcal R}   \mathcal D_{U(N)}(\theta_1, \dots, \theta_N) \; d\theta_1\dots d\theta_N    \\
& = \binom{N}{k} \frac{1}{N!(2\pi)^N} \int_{-\pi}^{\pi} \left[ \int\dots\int_{\mathcal R'} \mathcal  D_{SU(N)}(\lambda_1, \dots , \lambda_{N-1}) \; d\lambda_1 \dots d\lambda_{N-1} \right] d\lambda_N \\
& =  \binom{N}{k} \frac{1}{N!(2\pi)^{N-1}} \int\dots\int_{\mathcal R'}   \mathcal  D_{SU(N)}(\lambda_1, \dots , \lambda_{N-1}) \; d\lambda_1 \dots d\lambda_{N-1}
\end{align*}
If we write down explicitly the conditions of $\mathcal R'$ into the integral, it becomes
\begin{multline*}
\binom{N}{k} \frac{1}{N!(2\pi)^{N-1}} \int_{\mathcal J^k}\int_{([-\pi, \pi) \smallsetminus \mathcal J)^{N-1-k}}  \mathcal  D_{SU(N)}(\lambda_1, \dots , \lambda_{N-1}) \times \\
\times \chi_{[-\pi, \pi) \smallsetminus \mathcal J}(-\lambda_1-\ldots - \lambda_{N-1} \text{ mod } 2\pi) \; d\lambda_1 \dots d\lambda_{N-1}
\end{multline*}
where $\chi_I(x)$ denotes the characteristic function (also known as the indicator function) of the interval $I$. Now, because 
\[  \chi_I(x) = \int_{I} \delta(y-x) dy  \]
we can re-introduce into our integral the variable $\lambda_N$ (that was previously integrated out) and obtain
\begin{align*}
&\binom{N}{k} \frac{1}{N!(2\pi)^{N-1}} \int_{\mathcal J^k}\int_{([-\pi, \pi) \smallsetminus \mathcal J)^{N-k}}  \mathcal  D_{U(N)}(\lambda_1, \dots , \lambda_N) \cdot \delta( \lambda_1 + \ldots + \lambda_N \text{ mod } 2\pi ) \; d\lambda_1 \dots d\lambda_N  \\
& = \binom{N}{k} \int_{\mathcal J^k}\int_{([-\pi, \pi) \smallsetminus \mathcal J)^{N-k}}  \mathcal  P_{SU(N)}(\lambda_1, \dots , \lambda_N) \;  d\lambda_1 \dots d\lambda_N  \\
  & = \binom{N}{k} \text{Pr }[ \lambda_1, \dots , \lambda_k \in \mathcal J \> \text{ and } \> \lambda_{k+1}, \dots , \lambda_N\in [-\pi, \pi)\smallsetminus \mathcal J \text{ and } \lambda_1 + \ldots + \lambda_N = 0 \text{ mod } 2\pi ] \\
&= \text{Pr }[\text{exactly } k \text{ eigenvalues of a SU}(N) \text{ matrix lie in a SU}(N) \text{ Gram interval}]
  \end{align*}
as required. In the last step we have used, as in the beginning, the fact that for any $k$ the probability of having $k$ eigenvalues of a $\SU(N)$ matrix in a $\SU(N)$ Gram interval is the same, for any $\SU(N)$ Gram interval and any $k$ eigenvalues.
\end{proof}
\quad \\

In analogy with the quantity from \eqref{eq:eun}, we will denote the later probability of the above lemma by $ E_{SU(N)}(k, \mathcal J) $. \\

\begin{definition} 
We define the probability of having exactly $k$ unscaled eigenvalues of a random Haar-distributed  $\SU(N)$ matrix inside a $\SU(N)$ Gram interval $\mathcal J$ as
\begin{equation} \label{eq:esun}
  E_{SU(N)}(k, \mathcal J)  :=  \binom{N}{k} \int _{\mathcal J^k} \int_{([-\pi, \pi)\smallsetminus \mathcal J)^{N-k}}  \mathcal P_{SU(N)}(\theta_1, \ldots, \theta_N) \, d\theta_1 \ldots d\theta_N
\end{equation} \\
\end{definition} 

As we have just proved, $  E_{SU(N)}(k, \mathcal J)$ also represents the probability of finding $k$ eigenvalues of a random $\U(N)$ matrix in a $\U(N)$ Gram interval, and by using the above formula, we can compute it numerically for the same values of  $k$ and $N$ as in Table  \ref{tab:E_U(N)(k,J)}; this is presented in Table~\ref{tab:E_SU(N)(k,J)}.

\begin{table}[h] 
  \centering
\begin{tabular}{| r | c | c | c |}
\hline
 $N$ &  $E_{SU(N)}(0,\mathcal J)$  &  $E_{SU(N)}(1, \mathcal J)$  &  $E_{SU(N)}(2, \mathcal J)$  \\
\hline
 2  & & 1.000000 & \\
 3  & 0.023074 & 0.954844 & 0.021090 \\
 4  & 0.040362 & 0.920641 & 0.037630 \\
 5  & 0.067146 & 0.867251 & 0.064059 \\
 6  & 0.084764 & 0.832111 & 0.081483 \\
 7  & 0.097237 & 0.807224 & 0.093839 \\
 8  & 0.106532 & 0.788673 & 0.103058 \\
 9  & 0.113727 & 0.774310 & 0.110200 \\
10 & 0.119461 & 0.762860 & 0.115896 \\
11 & 0.124138 & 0.753520 & 0.120545 \\
12 & 0.128026 & 0.745755 & 0.124412 \\
13 & 0.131309 & 0.739197 & 0.127678 \\
14 & 0.134118 & 0.733586 & 0.130474 \\
15 & 0.136549 & 0.728730 & 0.132894 \\
16 & 0.138673 & 0.724486 & 0.135009 \\
17 & 0.140545 & 0.720746 & 0.136874 \\
18 & 0.142207 & 0.717424 & 0.138530 \\
19 & 0.143693 & 0.714455 & 0.140011 \\
20 & 0.145029 & 0.711785 & 0.141343 \\
21 & 0.146237 & 0.709371 & 0.142547 \\
\hline
\end{tabular}
\medskip
\caption{$E_{SU(N)}(k,\mathcal J)$  for $k=0,1,2$ and $N=2,\ldots, 21$} \label{tab:E_SU(N)(k,J)} \label{tab:esun}
\end{table}

On the one hand, if we compare those values with the corresponding values from Table~\ref{tab:E_U(N)(k,J)}, we see that $E_{SU(N)}(k,\mathcal J)$ has a different rate of convergence, in the sense that it doesn't converge to its large $N$ limit as fast as $E_{U(N)}(k, J)$. On the other hand, if we look at the entries in Table~\ref{tab:GramStatsLowHeight}, we notice that for each $k$ and $N$, $E_{SU(N)}(k,\mathcal J)$ does provide a good approximation for  $G_{M_N , M_{N+1}}(k)$, in accordance with Conjecture 1.

The observation that these two quantities appear to have the same rate of convergence hints at the possibility that they should also have the same asymptotic limit, which leads us to put forward the following alternative to Conjecture 2:

\begin{conjecture}
For any $k\in \mathbb N \cup \{0\}$
\[  \lim_{N\to\infty} E_{SU(N)}(k, \mathcal J)  = \lim_{M\to\infty}  G_{0,M}(k)  \]
\end{conjecture}

Finally, we remark that although $ E_{SU(N)}(k, \mathcal J) $ converges at a slower rate compared to  $E_{U(N)}(k, J)$, this does not necessarily imply that they don't tend to the same limit; in order to clarify whether this is or not the case, we have to obtain a more explicit formula that describes how $ E_{SU(N)}(k, \mathcal J) $  depends on the matrix size for finite, but arbitrarily large $N$. This will be the topic of the next section.

\section{Explicit formulas for $E_{\SU(N)}(k, \mathcal J)$}

\subsection{$\U(N)$ $n$-level densities and generating function}

One of the main quantities of interest needed for studying how the eigenvalues of random matrices are distributed  is the $n$-level density, defined as follows:

\begin{definition} The $\U(N)$ n-level density  (also called the the $n$-point correlation function) for the eigenangles of a unitary matrix is defined as
 \begin{align*}
 R^n_{U(N)}(\theta_1, \ldots, \theta_n) & := \frac{N!}{(N-n)!} \int_{[-\pi, \pi)^{N-n}} \mathcal P_{U(N)}(\theta_1, \ldots, \theta_N) \; d\theta_{n+1} \ldots d\theta_N  \\
						  & \quad = \frac{1}{(N-n)!(2\pi)^N} \int_{[-\pi, \pi)^{N-n}} \mathcal D_{U(N)}(\theta_1, \ldots, \theta_N) \; d\theta_{n+1} \ldots d\theta_N
\end{align*}
where $n=1,\ldots, N$. It represents the probability of finding an eigenangle (regardless of labeling) around each of the points $\theta_1,\ldots,\theta_n$, and with all the other eigenangles being integrated out. 
\end{definition}

In order to go from the $n$-level density to the number of eigenvalues in an interval, we first note that the Dyson product can be expanded in the form of a trigonometric Fourier series

\begin{align*}
\mathcal D_{U(N)}(\theta_1,\ldots,\theta_N) &= \prod_{1\leq j<k\leq N} |e^{i\theta_j}-e^{i\theta_k}|^2\\
& = \prod_{1\leq j < k\leq N} [ 2 - 2\cos(\theta_j - \theta_k) ] \\
& = \sum_{ \substack{ (j_1, \ldots, j_N)\in \mathbb Z^N \\ j_1+\ldots+j_N=0 \\ |j_l|<N  } } c_{j_1,\ldots , j_N} \cdot \cos(j_1\theta_1 + \ldots + j_N\theta_N)
\end{align*}
 where  $ c_{j_1,\ldots , j_N}\neq 0$  if and only if the corresponding $j_1, \dots, j_N \in \mathbb Z$ satisfy the conditions $j_1+\ldots+j_N=0$ and $|j_l|<N$ for all $l=1,\ldots,N$. This implies that the  $n$-level density becomes
\begin{multline*}
 R^n_{U(N)}(\theta_1, \ldots, \theta_n) = \\
  \frac{1}{(N-n)!(2\pi)^N}  \sum_{ \substack{ (j_1, \ldots, j_N)\in \mathbb Z^N \\ j_1+\ldots+j_N=0 \\ |j_l| < N  } } \left[ c_{j_1,\ldots , j_N} \int_{[-\pi, \pi)^{N-n}} \cos(j_1\theta_1 + \ldots + j_N\theta_N) \; d\theta_{n+1} \ldots d\theta_N \right] 
 \end{multline*}
The above expression can be simplified with the following observation: for each cosine term in the sum, if any of the coefficients $j_m$ corresponding to the integration variable $\theta_m $ ($m=n+1, \dots, N$) is non-zero, then the integral over that term is zero, since the cosine is being integrated over a full period.  This implies that the only the terms in the sum that give non-zero contributions to the $n$-level density are those terms for which $j_{n+1} = \ldots = j_N = 0$ or, in other words, the terms that depend only on the eigenangles $\theta_1, \dots, \theta_n$, together with the constant term $c_{0,\ldots, 0}$.

In the particular case of the  $\U(N)$ one-level density, because there is no cosine term in $ \mathcal D_{U(N)}(\theta_1,\ldots,\theta_N)$ that depends only on $\theta_1$, the constant term $c_{0,\ldots, 0}$ is the one that makes the only non-zero contribution. It is known that $c_{0,\ldots, 0}=N!$ \cite{good}, \cite{gunson}, \cite{wilson}, \cite{zeilberger}, which implies that
\[ \mathcal D_{U(N)}(\theta_1,\ldots,\theta_N)  = N! +\ldots \]
 and, therefore
 \[  R^1_{U(N)}(\theta_1)  = \frac{N}{2\pi} \]

Similarly, for the $\U(N)$ two-level density, we get a contribution from the constant term, together with the terms that depend only on $\theta_1, \theta_2$, which are
\[ \mathcal D_{U(N)}(\theta_1, \ldots, \theta_N) = N! - (N-2)!\cdot 2 \sum_{a=1}^{N-1} (N-a)\cos(a\theta_1 - a\theta_2) + \ldots \]
and, as a consequence
\[ R^2_{U(N)}(\theta_1, \theta_2) = \frac{1}{(2\pi)^2} \left[ (N-1)N - 2\sum_{a=1}^{N-1} (N-a) \cos(a\theta_1-a\theta_2) \right]  \]

With this approach, exact knowledge of the $\U(N)$ $n$-level density requires exact knowledge of the coefficients of all terms in $\mathcal D_{U(N)}(\theta_1, \ldots, \theta_N)$ that depend only on $\theta_1, \ldots, \theta_n$, but in practice, these coefficients become more and more difficult to compute explicitly as $n$ increases. 

\begin{remark}
In the case of the $\U(N)$ group, there is also another, more elegant approach for computing the $n$-level densities, which does not require knowledge of any of the above coefficients; this method is described, for example, in \cite{conrey}, \cite{forrester}, \cite{fyodorov}. However, that approach can not be applied to $\SU(N)$ matrices, essentially because $\mathcal P_{SU(N)}(\theta_1, \ldots, \theta_N)$ does not appear to have a determinant form.
\end{remark}

In addition to the $n$-level density, the other important quantity required for the study of eigenvalue distribution is the so-called generating function.

\begin{definition} If $J\subset [-\pi, \pi)$ is an arbitrary interval on the unit circle of any length and $ E_{\U(N)}(k, J)$ is defined as in $\eqref{eq:eun}$, the $\U(N)$ generating function is given by
\[ \mathcal E_{U(N)}(z, J) := \sum_{k=0}^N (1+z)^k E_{U(N)}(k, J)  \]
\end{definition}

If the generating function is know, one can immediately recover the desired probabilities through repeated differentiation
\begin{equation}   \label{eq:eun2}
E_{U(N)}(k, J) = \frac{1}{k!} \left. \left( \frac{d^k}{d z^k}  \mathcal E_{U(N)}(z, J) \right) \right|_{z=-1}  
\end{equation}

It can be shown that the generating function can also be expressed in terms of all the $n$-level densities as a sum
\begin{equation}  \label{eq:ungenfunc}
\mathcal E_{U(N)}(z, J) =  1 + \sum_{n=1}^N \frac{z^n}{n!} \int_{J^n} R^n_{U(N)}(\theta_1, \ldots, \theta_n) ]\, d\theta_1\ldots d\theta_n
\end{equation}
In the case of $\U(N)$ (which does not hold for $\SU(N)$) an identity that is due to Gram himself shows that this sum is equal to a $N\times N$ determinant
\[  1 + \sum_{n=1}^N \frac{z^n}{n!} \int_{J^n} R^n_{U(N)}(\theta_1, \ldots, \theta_n) \, d\theta_1\ldots d\theta_n  =  \det_{N\times N}\left[ I_N + \frac{z}{2\pi}\left( \int_J e^{i(j-k)\theta} d\theta \right)_{1\leq j,k\leq N} \right]   \]

Putting together the above equations, we obtain an analytic formula for each  $E_{U(N)}(k, J)$ as the $k$-th order derivative of this determinant
\[ E_{U(N)}(k, J) = \frac{1}{k!} \cdot \left. \frac{d^k}{d z^k}  \left\{ \det_{N\times N}\left[ I_N + \frac{z}{2\pi}\left( \int_J e^{i(j-k)\theta} d\theta \right)_{1\leq j,k\leq N} \right] \right\} \right|_{z=-1}  \]
\\
\subsection{$\SU(N)$ $n$-level densities and generating function}

We now proceed to extend the notions defined above to $\SU(N)$ matrices, in order to derive a formula for $E_{SU(N)}(k, \mathcal J)$.
 \begin{definition} We define the $\SU(N)$ $n$-level density of a special unitary matrix as
\begin{multline} \label{eq:sunlevel}
 R^n_{SU(N)}(\theta_1, \ldots, \theta_n) := \frac{N!}{(N-n)!} \int_{[-\pi,\pi)^{N-n}} \mathcal P_{SU(N)}(\theta_1, \ldots, \theta_N) \; d\theta_{n+1} \ldots d\theta_N \\
= \frac{1}{(N-n)!(2\pi)^{N-1}} \int_{[-\pi,\pi)^{N-n-1}} \mathcal D_{SU(N)}(\theta_1, \ldots, \theta_{N-1}) \; d\theta_{n+1} \ldots d\theta_{N-1} 
\end{multline}
This time, we can only have $n=1, \ldots,  N-1$, because if the values of $N-1$ eigenangles are given,  then the $N$-th one is already determined by the restriction $\theta_1+\ldots+\theta_N=0 \pmod{2\pi}$, so it can not be assigned to an arbitrary value.
\end{definition}

As in the $\U(N)$ case, knowledge of the $\SU(N)$ $n$-level density is based on on knowledge of the coefficients of all terms in $\mathcal D_{SU(N)}(\theta_1, \dots, \theta_{N-1})$ that depend only on $\theta_1, \ldots, \theta_n$ which, in turn, depends on knowledge of the corresponding terms in $\mathcal D_{U(N)}(\theta_1, \dots, \theta_N)$ before the substitution $\theta_N = -\theta_1 - \ldots - \theta_{N-1}$ was made. \\

For example, to obtain the $\SU(N)$ one-level density, we first remark that
\[ \mathcal D_{U(N)}(\theta_1,\ldots,\theta_N)  = N! + 2(-1)^{N-1}(N-1)!\cos((N-1)\theta_1 - \theta_2 - \ldots - \theta_N) + \ldots  \]
which, after the substitution $\theta_N = -\theta_1 - \ldots - \theta_{N-1}$ becomes
\[ \mathcal D_{SU(N)}(\theta_1, \dots , \theta_{N-1}) = N! + 2(-1)^{N-1}(N-1)!\cos( N\theta_1) + \ldots \]

So we see that, unlike $\mathcal D_{U(N)}(\theta_1,\ldots,\theta_N) $, the $\mathcal D_{SU(N)}(\theta_1, \dots , \theta_{N-1}) $ does contain a term that depends only on $\theta_1$.
And because $ \mathcal D_{SU(N)}(\theta_1, \dots , \theta_{N-1}) $ has the same constant term as $\mathcal D_{U(N)}(\theta_1,\ldots,\theta_N) $, we can express the $\SU(N)$ one-level density as the $\U(N)$ one-level density plus the integral of that additional term
\begin{align*}
R^1_{SU(N)}(\theta_1) &= \frac{N}{2\pi}+  \frac{ 2(-1)^{N-1}\cos (N\theta_1)}{2\pi} \\
&=  R^1_{U(N)}(\theta_1) +  \frac{ 2(-1)^{N-1}\cos (N\theta_1)}{2\pi}
\end{align*}

Similarly, in order to compute the $\SU(N)$ two-level density, it can be shown that  $ \mathcal D_{SU(N)}(\theta_1, \dots , \theta_{N-1}) $ contains all the terms in $\mathcal D_{U(N)}(\theta_1,\ldots,\theta_N) $ that depended only on $\theta_1,\theta_2$ plus several additional ones, given by

\begin{align*}
  \mathcal D_{SU(N)}&(\theta_1, \ldots, \theta_{N-1} ) = \\
   &  \left[ N! - (N-2)!  \sum_{a=1}^{N-1} 2(N-a) \, \cos(a\theta_1 - a\theta_2) \right]   +  4(N-2)!\cos(N\theta_1 + N\theta_2) + \\
 & - 2(N-2)! [ \cos((N+1)\theta_1 + (N-1)\theta_2) +\cos((N-1)\theta_1 + (N+1)\theta_2) ] +  \\
 & + 2(-1)^{N-1}(N-1)!  \cos(N\theta_1) +  2(-1)^{N-1}(N-1)! \cos(N\theta_2) + \\
& + 4 (-1)^{N-2}  (N-2)! \mathop{ \sum_{k_1=1}^{N-1} \sum_{k_2=1}^{N-1}}_{k_1 + k_2 = N} \cos( k_1 \theta_1 + k_2 \theta_2 ) + \ldots 
\end{align*}

This, again, leads to the fact that the $\SU(N)$ two-level density can be written in terms of the $\U(N)$ two-level density plus an additional contribution
\begin{align*}
   R^2_{SU(N)}(\theta_1, \theta_2) &= R^2_{U(N)}(\theta_1, \theta_2)  +  \frac{2}{(2\pi)^2}  [  2 \cos(N\theta_1+N\theta_2) - \\
						& \quad - \cos((N+1)\theta_1+(N-1)\theta_2)  - \cos((N-1)\theta_1+(N+1)\theta_2) +    \\
						& \quad + (-1)^{N-1}(N-1)\cos(N\theta_1) + (-1)^{N-1}(N-1)\cos(N\theta_2) + \\
						& \quad + 2(-1)^{N-2} \mathop{ \sum_{k_1 = 1}^{N-1} \sum_{k_2=1}^{N-1}}_{k_1 + k_2=N} \cos(k_1 \theta_1 + k_2 \theta_2)  ] 
\end{align*}

The above observations can be summarized and generalized in the following way: let $n=1, \ldots, N-1$ be fixed. All the terms in $\mathcal D_{U(N)}(\theta_1,\ldots,\theta_N) $  that depend only on $\theta_1 , \ldots , \theta_n $ also appear unchanged in  $ \mathcal D_{SU(N)}(\theta_1, \dots , \theta_{N-1}) $ and this, in turn, implies that all the terms of the $\U(N)$ $n$-level density are always included among the terms of the $\SU(N)$ $n$-level density. Furthermore, for any $m=1,\ldots, n$ there are other terms in $\mathcal D_{U(N)}(\theta_1,\ldots,\theta_N) $ that depend on all the eigenangles, such as
\[ \cos( j_1\theta_1 + \ldots + j_n \theta_n - m\theta_{n+1} -\ldots  - m\theta_N ) ,\quad \text{ with } \quad j_1 + \ldots + j_n=m(N-n) \]
which, after the substitution  $\theta_N = -\theta_1 - \ldots - \theta_{N-1}$, depend only on $\theta_1, \ldots, \theta_n$
\[ \cos( (j_1+m)\theta_1 + \ldots + (j_n+m) \theta_n ) \]
In other words, these are terms that do not appear in the $\U(N)$ $n$-level density (those are the terms for which $m=0$), but contribute to the $\SU(N)$ $n$-level density; we will denote all these extra terms collectively by
\begin{multline*}
 X^n (\theta_1, \ldots, \theta_n)  := \\
\frac{2}{(2\pi)^n} \sum_{m=1}^n \sum_{ \substack{ (j_1, \ldots , j_n) \in \mathbb Z^n  \\ j_1 + \ldots + j_n =m(N-n)  \\  -N < j_1, \ldots , j_n < N} } \cos((j_1 + m)\theta_1 + \ldots + (j_n+m)\theta_n) \cdot c_{j_1+m , \ldots, j_n+m}
\end{multline*}
We will also relabel the indexes as $ k_i = j_i +m $ , so that
\[ X^n  (\theta_1, \ldots, \theta_n) = \frac{2}{(2\pi)^n} \sum_{m=1}^n \sum_{ \substack{ (k_1, \ldots , k_n) \in \mathbb Z^n  \\ k_1 + \ldots + k_n =mN  \\ -N+m < k_1, \ldots , k_n < N+m} } \cos(k_1 \theta_1 + \ldots + k_n\theta_n) \cdot c_{k_1, \ldots, k_n} \]

We conclude the discussion about $\SU(N)$ $n$-level densities by restating that we have
\[ R^n_{SU(N)}(\theta_1, \ldots, \theta_n) =  R^n_{U(N)}(\theta_1, \ldots, \theta_n)  + X^n(\theta_1, \ldots, \theta_n) , \quad n=1, \ldots,  N-1  \]

For later notional simplicity, we also extend the definition of $X_{SU(N)}^n$ to the cases $n=0$ and $n=N$ by
\begin{equation} \label{eq:xn}
X^n(\theta_1, \ldots, \theta_n)  =
\begin{cases}
		0 , & \text{ if } n=0 \\
		R_{SU(N)}^n(\theta_1, \ldots, \theta_n) - R_{U(N)}^n(\theta_1, \ldots, \theta_n) , & \text{ if }n=1,\ldots, N-1 \\
		N! \mathcal P_{SU(N)}(\theta_1, \ldots, \theta_N) - R^N_{U(N)}(\theta_1, \ldots, \theta_N), & \text{ if }  n=N \\
\end{cases} 
\end{equation}
\\
We can now continue the analogy with $\U(N)$  by studying the SU(N) generating function, first for an arbitrary interval, and then for a $\SU(N)$ Gram interval.

\begin{definition}  If $J\subset [-\pi, \pi)$ is an arbitrary interval on the unit circle of any length and $ E_{SU(N)}(k, J)$ is defined as in $\eqref{eq:esun}$, the $\SU(N)$ generating function is given by
\[ \mathcal E_{SU(N)}(z, J) := \sum_{k=0}^N (1+z)^k E_{SU(N)}(k, J) \]
\end{definition}

As in the $\U(N)$ case, we can obtain a formula for $E_{SU(N)}(k, J)$  in terms of the the generating function 
\begin{equation}  \label{eq:esun2}
E_{SU(N)}(k, J) = \frac{1}{k!} \left. \left( \frac{d^k}{d z^k}  \mathcal E_{SU(N)}(z, J) \right) \right|_{z=-1} 
\end{equation}

We will now use the fact that each $\SU(N)$ $n$-level density depends on the corresponding $\U(N)$ $n$-level density to prove that the $\SU(N)$ generating function can also be expressed in terms of the $\U(N)$ generating function. \\

\begin{theorem}
 For an arbitrary interval $J\subset [-\pi, \pi)$ on the unit circle of any length, we have that
\begin{equation} \label{eq:mainth}
\mathcal E_{SU(N)}(z, J) = \mathcal E_{U(N)}(z, J) + \sum_{n=1}^N \frac{z^n}{n!} \int_{J^n} X_{SU(N)}^n(\theta_1, \ldots, \theta_n) \; d\theta_1\ldots d\theta_n
\end{equation}
where $X^n_{SU(N)}$ is defined in \eqref{eq:xn}.
\end{theorem}

\begin{proof}
First, we denote the SU($N$) Haar measure by
 \[ d\mu_{SU(N)} = \mathcal P_{SU(N)} (\theta_1, \ldots, \theta_N) \; d\theta_1 \ldots d\theta_N  \]
 and show that
\[   \int_{SU(N)} (1 + z\chi_J(\theta_1)) \ldots (1 + z\chi_J(\theta_N)) \; d\mu_{SU(N)} = \sum_{k=0}^N (1+z)^k E_{SU(N)}(k, J) \]
 (where, again,  $\chi_J(\theta)$ is the characteristic function)

For $k=0,\ldots, N$, let $P_k\subset $ $\SU(N)$ be the subset of all $\SU(N)$ matrices with exactly $k$ eigenangles in $J$. These sets are all pairwise disjoint and their union is SU$(N)$, so they form a partition of $\SU(N)$. This implies that
\begin{align*}
& \int_{SU(N)} (1 + z\chi_J(\theta_1)) \ldots (1 + z\chi_J(\theta_N)) \; d\mu_{SU(N)}  \\
& =  \int_{P_0 \cup \ldots \cup P_N} (1 + z\chi_J(\theta_1)) \ldots (1 + z\chi_J(\theta_N)) \; d\mu_{SU(N)}  \\
& = \sum_{k=0}^N \int_{P_k} (1 + z\chi_J(\theta_1)) \ldots (1 + z\chi_J(\theta_N)) d\mu_{SU(N)}  \\
& = \int_{P_0} d\mu_{SU(N)} + (1+z) \int_{P_1} d\mu_{SU(N)} + \ldots + (1+z)^N\int_{P_N} d\mu_{SU(N)}
\end{align*}
and, by definition, $E_{SU(N)}(k, J)$ represents the measure of the set of $\SU(N)$ matrices which have precisely $k$ eigenangles in $J$, so
\[ E_{SU(N)}(k, J) =  \int_{P_k} d\mu_{SU(N)}  \]
That is, we have shown that
\begin{align*}
 \mathcal E_{SU(N)}(z, J) & = \int_{SU(N)} (1 + z\chi_J(\theta_1)) \ldots (1 + z\chi_J(\theta_N)) \; d\mu_{SU(N)}  \\
				&  = \int_{[-\pi. \pi)^N} (1 + z\chi_J(\theta_1)) \ldots (1 + z\chi_J(\theta_N)) \;  \mathcal P_{SU(N)}(\theta_1, \ldots, \theta_N) \; d\theta_1 \ldots d\theta_N 
\end{align*}
After we open all the brackets, this becomes
\[ \mathcal E_{SU(N)}(z, J) = 1 + \sum_{n=1}^N z^n \int_{[-\pi, \pi)^N}  h_n(\chi_J(\theta_1), \ldots, \chi_J(\theta_N)) \;  \mathcal P_{SU(N)}(\theta_1, \ldots, \theta_N) \; d\theta_1 \ldots d\theta_N  \]
where the $h_n$'s are elementary symmetric polynomials in $\chi_J(\theta_1), \ldots, \chi_J(\theta_N)$
\[ \begin{cases}
	h_1 = \chi_J(\theta_1) + \ldots + \chi_J(\theta_N) \medskip    \\
	h_2 = \displaystyle \sum_{1\leq i<j\leq N} \chi_J(\theta_i)\chi_J(\theta_j) \\
	\dots\dots\dots\dots\dots\dots\dots\dots \\
	h_N = \chi_J(\theta_1) \ldots \chi_J(\theta_N)
\end{cases} \]
Because each $h_n$ has $\displaystyle \binom{N}{n}$ terms, and $\mathcal P_{SU(N)}(\theta_1, \ldots, \theta_N) $
is invariant under the permutation of any of its arguments, the $\SU(N)$ generating function becomes
\begin{align*}
& \mathcal E_{SU(N)}(z, J)  \\
& = 1 + \sum_{n=1}^N z^n \frac{N!}{n!(N-n)!} \int_{[-\pi, \pi)^N}  \chi_J(\theta_1)\ldots \chi_J(\theta_n) \; \mathcal P_{SU(N)}(\theta_1, \ldots, \theta_N)  \; d\theta_1 \ldots d\theta_N = \\
& \\
& = 1 + \sum_{n=1}^N \frac{z^n}{n!} \; \frac{N!}{(N-n)!} \int_{J^n} \int_{[-\pi,\pi)^{N-n}} \mathcal P_{SU(N)} (\theta_1, \ldots, \theta_N) \; d\theta_1 \ldots d\theta_N = \\
& \\
& = 1 + \sum_{n=1}^N \frac{z^n}{n!} \int_{J^n} \left[ \frac{N!}{(N-n)!} \int_{[-\pi,\pi)^{N-n}} \mathcal P_{SU(N)} (\theta_1, \ldots, \theta_N) d\theta_{n+1} \ldots d\theta_N \right]  d\theta_1\ldots d\theta_n = \\
& \\
& = 1 + \sum_{n=1}^{N-1} \frac{z^n}{n!} \int_{J^n} R^n_{SU(N)}(\theta_1, \ldots, \theta_n) \; d\theta_1\ldots d\theta_n + \frac{z^N}{N!} \int_{J^N} N! \mathcal P_{SU(N)} (\theta_1, \ldots, \theta_N) \; d\theta_1\ldots d\theta_N
\end{align*}
where in the last line we use \eqref{eq:sunlevel} to obtain $R^n_{\SU(N)}$ for $n=1,\dots,N-1$. Now if we apply \eqref{eq:xn}, we find that $\mathcal E_{SU(N)}(z, J)$ equals
\begin{align*}
= & 1 + \sum_{n=1}^{N-1} \frac{z^n}{n!} \int_{J^n} [ R^n_{U(N)}(\theta_1, \ldots, \theta_n) + X^n(\theta_1, \ldots, \theta_n) ] \; d\theta_1\ldots d\theta_n + \\
& +  \frac{z^N}{N!} \int_{J^N} [ R^N_{U(N)} (\theta_1, \ldots, \theta_N) + X^N(\theta_1, \ldots, \theta_N) ] \; d\theta_1\ldots d\theta_N  
\end{align*}
By re-grouping the terms, this equals
\[ = 1 + \sum_{n=1}^N \frac{z^n}{n!} \int_{J^n} R^n_{U(N)}(\theta_1, \ldots, \theta_n) \; d\theta_1\ldots d\theta_n + \sum_{n=1}^{N} \frac{z^n}{n!} \int_{J^n} X^n(\theta_1, \ldots, \theta_n) \; d\theta_1\ldots d\theta_n \]
and noting the first term is the $\U(N)$ generating function \eqref{eq:ungenfunc}, we have shown that
\[ \mathcal E_{SU(N)}(z, J) = \mathcal E_{U(N)}(z, J)  + \sum_{n=1}^N \frac{z^n}{n!} \int_{J^n} X^n(\theta_1, \ldots, \theta_n) \; d\theta_1\ldots d\theta_n \]
which was our desired result.  \\
\end{proof}

The above formula can be used, together with \eqref{eq:eun2} and \eqref{eq:esun2}, to derive a similar relation between $ E_{SU(N)}(k, J)$ and $ E_{U(N)}(k,J) $
\begin{corollary}
For any $k=0,\ldots, N$ and any $J\subset [-\pi, \pi)$, we have
\[ E_{SU(N)}(k, J) = E_{U(N)}(k,J) + \sum_{n=k}^N  \frac{(-1)^{n-k}}{k!(n-k)!} \int_{J^n} X^n(\theta_1, \ldots, \theta_n) \; d\theta_1\ldots d\theta_n  \]
\end{corollary}

\begin{proof}
\begin{align*}
E_{SU(N)}(k, J) &= \frac{1}{k!} \left. \left( \frac{d^k}{d z^k}  \mathcal E_{SU(N)}(z, J) \right) \right|_{z=-1} \\
		     &=  \frac{1}{k!} \left. \left( \frac{d^k}{d z^k}  \mathcal E_{U(N)}(z, J) \right) \right|_{z=-1}  + \frac{1}{k!} \left. \left( \frac{d^k}{d z^k}  \sum_{n=1}^N \frac{z^n}{n!} \int_{J^n} X^n(\theta_1, \ldots, \theta_n) \; d\theta_1\ldots d\theta_n \right) \right|_{z=-1} \\
		     & = E_{U(N)}(k,J) + \sum_{n=k}^N  \frac{(-1)^{n-k}}{k!(n-k)!} \int_{J^n} X^n(\theta_1, \ldots, \theta_n) \; d\theta_1\ldots d\theta_n
\end{align*}
\end{proof}

\subsection{Application to $\SU(N)$ Gram intervals and conclusions}

As previously mentioned, up to this point $J\subset [-\pi, \pi)$ was allowed to be  an arbitrary interval on the unit circle of any length. We conclude this section with a brief discussion on how $  \mathcal E_{SU(N)}(z, J) $ and $E_{SU(N)}(k, J)$ are affected in the particular case when $J$ is taken to be a $\SU(N)$ Gram interval, for example
\[  \mathcal J = \left[-\pi,-\pi + \frac{2\pi}{N} \right) \]

First, it can be easily seen that the integral of the $X^1_{SU(N)}(\theta_1)$ term over $\mathcal J$ will always be zero
\[  X^1(\theta_1) =  \frac{2(-1)^{N-1}\cos (N\theta_1)}{2\pi}     \quad \Rightarrow \quad \int_{\mathcal J} X^1(\theta_1) \; d\theta_1 =  0 \]
So in this case, the first non-zero term of the sum in \eqref{eq:mainth} is at $n=2$
\[ \mathcal E_{SU(N)}(z,\mathcal J) = \mathcal E_{U(N)}(z, \mathcal J) + \sum_{n=2}^N \frac{z^n}{n!} \int_{\mathcal J^n} X^n(\theta_1, \ldots, \theta_n) \; d\theta_1\ldots d\theta_n   \]

We recall from the formula for the $\SU(N)$ two-level density that the $ X^2(\theta_1, \theta_2)$ term is explicitly given by
\begin{align*}
   X^2(\theta_1, \theta_2) &=  \frac{2}{(2\pi)^2}  [  2 \cos(N\theta_1+N\theta_2)   - \cos((N+1)\theta_1+(N-1)\theta_2) \\
						& \quad - \cos((N-1)\theta_1+(N+1)\theta_2)   + (-1)^{N-1}(N-1)\cos(N\theta_1) +\\
						&\quad + (-1)^{N-1}(N-1)\cos(N\theta_2) + 2(-1)^{N-2} \mathop{ \sum_{k_1 = 1}^{N-1} \sum_{k_2=1}^{N-1}}_{k_1 + k_2=N} \cos(k_1 \theta_1 + k_2 \theta_2)  ]
\end{align*}
This can be used to show that its integral over a pair of $\SU(N)$ Gram intervals is given by
\[ \int_{\mathcal J^2} X^2 (\theta_1, \theta_2) \, d\theta_1 d\theta_2  =  \frac{4}{\pi^2} \left[ \frac{1}{N^2-1} \left(\sin\frac{\pi}{N} \right)^2  -  \mathop{ \sum_{k_1 = 1}^{N-1} \sum_{k_2=1}^{N-1}}_{k_1 + k_2=N}  \frac{1}{k_1 \, k_2} \sin\left( \frac{k_1 \pi}{N} \right) \sin\left( \frac{k_2\pi}{N} \right) \right] \]

We remark that the double sum can be re-expressed as a single sum
\begin{align*}
\mathop{ \sum_{k_1 = 1}^{N-1} \sum_{k_2=1}^{N-1}}_{k_1 + k_2=N}  \frac{1}{k_1 \, k_2} \sin\left( \frac{k_1 \pi}{N} \right) \sin\left( \frac{k_2\pi}{N} \right) &=  \sum_{k=1}^{N-1} \frac{1}{k(N-k)} \sin \left( \frac{k\pi}{N}\right)  \sin\left(\pi-\frac{k\pi}{N}\right) \\
&= \frac{2}{N}\sum_{k=1}^{N}\frac{1}{k} \left(\sin\frac{k\pi}{N} \right)^2 
\end{align*}
and using Euler-Maclaurin summation, this sum can be approximated by an integral, 
\begin{align*}
 \sum_{k=1}^{N}\frac{1}{k} \left(\sin\frac{k\pi}{N} \right)^2 &= \int_1^N \frac{1}{x} \left(\sin\frac{x\pi}{N} \right)^2 dx + \mathcal O\left( \frac{1}{N^2} \right) \\
& = \int_{\frac{1}{N}}^1 \frac{(\sin y\pi )^2}{y} dy + \mathcal O \left( \frac{1}{N^2} \right)  
\end{align*}
The integral, in turn, can be expressed in terms of the  cosine integral function Ci$(x)$
\[  \int_{\frac{1}{N}}^1 \frac{(\sin y\pi )^2}{y} dy = \frac{1}{2} \left[ \text{Ci}\left( \frac{2\pi}{N} \right) - \text{Ci}(2\pi) - \log\left( \frac{1}{N} \right) \right] \]
where Ci$(x)$ is defined as
\[ \text{Ci}(x) := -\int_{x}^\infty \frac{\cos t}{t}dt = \gamma + \log x + \sum_{n = 1}^\infty \frac{(-1)^n x^{2n}}{2n(2n)!} \]
This implies that the above sum is equal to
\[ \sum_{k=1}^{N}\frac{1}{k} \left(\sin\frac{k\pi}{N} \right)^2 = \frac{1}{2} [ \gamma + \log(2\pi) - \text{Ci}(2\pi)  ] + \mathcal O\left( \frac{1}{N^2} \right)\]
while the double sum becomes
\[ \mathop{ \sum_{k_1 = 1}^{N-1} \sum_{k_2=1}^{N-1}}_{k_1 + k_2=N}  \frac{1}{k_1 \, k_2} \sin\left( \frac{k_1 \pi}{N} \right) \sin\left( \frac{k_2\pi}{N} \right) =  \frac{1}{N} [ \gamma + \log(2\pi) - \text{Ci}(2\pi)  ] + \mathcal O\left( \frac{1}{N^3} \right)\]
In the case of the first term in the integral of $ X_{SU(N)}^2 (\theta_1, \theta_2) $, as $N$ increases, it has the order
\[  \frac{1}{N^2-1} \left(\sin\frac{\pi}{N} \right)^2  = \mathcal O\left( \frac{1}{N^4} \right) \]

Putting the previous results back together, we obtain that the contribution coming from the integral of $X_{SU(N)}^2 (\theta_1, \theta_2) $ is
\[ \int_{\mathcal J^2} X^2 (\theta_1, \theta_2) \, d\theta_1 d\theta_2  = -\frac{\alpha}{N}  + \mathcal O\left( \frac{1}{N^3} \right) \]
where
\[ \alpha := \frac{4 [ \gamma + \log(2\pi) - \text{Ci}(2\pi)  ]}{\pi^2} \approx 0.987944\dots \]

Numerical results suggest that the value of the integral of $ X^n(\theta_1, \ldots, \theta_n) $ over $\mathcal J^n$ decreases by several orders of magnitude as $n=3,\ldots, N$ increases, which implies that the integral of the $X^2(\theta_1, \theta_2)$ term gives the main error to $E_{SU(N)}(k,\mathcal J)$. However, according to Corollary 1, $X^2(\theta_1, \theta_2)$ appears in $E_{SU(N)}(k,\mathcal J)$ only for $k=0,1,2$. Keeping also in mind that in this case, the term $n=1$ of that sum is zero, we get the following approximations for these probabilities
\begin{align*}
 E_{SU(N)}(0,\mathcal J) &= E_{U(N)}(0, \mathcal J ) + \sum_{n=2}^N  \frac{(-1)^n}{0!n!} \int_{\mathcal J^n} X^n(\theta_1, \ldots, \theta_n) \; d\theta_1\ldots d\theta_n \\
				& = E_{U(N)}(0, \mathcal J ) + \frac{1}{2} \int_{\mathcal J^2} X^2 \, d\theta_1 d\theta_2 +\ldots  \\
&\approx   E_{U(N)}(0, \mathcal J) - \frac{\alpha}{2N}
\end{align*}
\begin{align*}
E_{SU(N)}(1, \mathcal J) &=  E_{U(N)}(1, \mathcal J ) + \sum_{n=2}^N  \frac{(-1)^{n-1}}{1!(n-1)!} \int_{\mathcal J^n} X^n(\theta_1, \ldots, \theta_n) \; d\theta_1\ldots d\theta_n \\
& = E_{U(N)}(1, \mathcal J ) - \int_{\mathcal J^2} X^2 (\theta_1, \theta_2)  \, d\theta_1 d\theta_2 +\ldots \\
& \approx E_{U(N)}(1, \mathcal J) + \frac{\alpha}{N} 
\end{align*}
\begin{align*}
E_{SU(N)}(2,\mathcal J) &=  E_{U(N)}(2, \mathcal J ) + \sum_{n=2}^N  \frac{(-1)^{n-2}}{2!(n-2)!} \int_{\mathcal J^n} X^n(\theta_1, \ldots, \theta_n) \; d\theta_1\ldots d\theta_n \\
& = E_{U(N)}(2, \mathcal J ) + \frac{1}{2} \int_{\mathcal J^2} X^2 (\theta_1, \theta_2)  \, d\theta_1 d\theta_2 + \ldots \\
&\approx E_{U(N)}(2,\mathcal J) - \frac{\alpha}{2N} 
\end{align*}

So, in conclusion, we note that in the large $N$ limit, the $\SU(N)$-probability of finding 0, 1, or 2 eigenvalues  in a $\SU(N)$ Gram interval converges to the $\U(N)$-probability of finding 0, 1, or 2 eigenvalues in an arbitrary interval of length $\frac{2\pi}{N}$, and we've estimated the rate of convergence. This supports a statement of Odlyzko \cite{odlyzko2}, who wrote that \\

\begin{quote}
``it seems reasonable to expect that at large heights the local distribution of the zeros will be independent of Gram points, which leads to the above assumption. In other words, the expectation is that at large heights, any grid of points spaced like the Gram points would exhibit similar behavior with respect to location of zeros.''
\end{quote}


\section*{Acknowledgements}

We would like to thank David Platt and the LMFDB Collaboration for providing us access to their data on the first 100 billion non-trivial zeros of the zeta function \cite{lmfdb}, without which the computation of the results from Table \ref{tab:GramStatsLowHeight} would not have been possible.

\quad \\


\begin{thebibliography}{99}

\bibitem{brent1} R. P. Brent, \textit{On the zeros of the Riemann zeta function in the critical strip},  Math. Comp. \textbf{33} (1979), 1361--1372 \medskip

\bibitem{brent2} R. P. Brent and J. van de Lune, \textit{On the zeros of the Riemann zeta function in the critical strip II}, Math. Comp. \textbf{39} (1982), 681--688 \medskip
    
\bibitem{conrey} J. B. Conrey, \textit{Notes on eigenvalue distributions for the classical compact groups}, in Recent Perspectives in Random Matrix Theory and Number Theory (Mezzadri and Snaith, eds), London Math. Soc. Lecture Note Ser., \textbf{322}, Cambridge Univ. Press,  Cambridge, (2005), 111--146 \medskip

\bibitem{dyson} F. J. Dyson, \textit{Statistical theory of the energy levels of complex systems I, II, and III} J. Math. Phys. \textbf{3} (1962) 140--175 \medskip

\bibitem{edwards} H. M. Edwards, \textit{Riemann's Zeta Function}, Pure Appl. Math. Ser. \textbf{58}, Academic Press, New York (1974) \medskip

\bibitem{forrester} P. J. Forrester, \textit{Spacing distributions in random matrix ensembles}, in Recent Perspectives in Random Matrix Theory and Number Theory (Mezzadri and Snaith, eds), London Math. Soc. Lecture Note Ser., \textbf{322}, Cambridge Univ. Press,  Cambridge, (2005), 279--308  \medskip

\bibitem{fyodorov} Y. V. Fyodorov, \textit{Introduction to the random matrix theory: Gaussian Unitary Ensemble and beyond}, in Recent Perspectives in Random Matrix Theory and Number
  Theory (Mezzadri and Snaith, eds), London Math. Soc. Lecture Note Ser., \textbf{322}, Cambridge Univ. Press,  Cambridge, (2005), 31--78  \medskip

\bibitem{fujii} A. Fujii, \textit{Gram's Law for the zeta zeros and the eigenvalues of Gaussian Unitary Ensembles}, Proc. Jap. Acad. Ser. A \textbf{63} (1987), 392--395 \medskip

\bibitem{good} I. J. Good, \textit{Short proof of a conjecture by Dyson}, J. Math. Phys. \textbf{11} (1970), 18--84 \medskip

\bibitem{gourdon} X. Gourdon, \emph{The $10^{13}$ first zeros of the Riemann zeta function, and zeros computation at very large height}, available at \url{http://numbers.computation.free.fr/Constants/Miscellaneous/zetazeros1e13-1e24.pdf}, (2004) \medskip

\bibitem{gram} J. P. Gram, \textit{Note sur les zeros de la function $\zeta(s)$ de Riemann}, Acta Math. \textbf{27} (1903), 289--304 \medskip

\bibitem{gunson} J. Gunson, \textit{Proof of a conjecture by Dyson in the statistical theory of energy levels}, J. Math. Phys. \textbf{3} (1962), 752--753 \medskip

\bibitem{hiai} F. Hiai et al., \textit{Inequalities related to free entropy derived from random matrix approximation}, Probab. Theory Relat. Fields \textbf{130} (2004), 199--221 \medskip

\bibitem{hutchinson} J. I. Hutchinson, \textit{On the roots of the Riemann zeta-function}, Trans. Amer. Math. Soc. \textbf{27} (1925), 49--60 \medskip

\bibitem{keatingsnaith} J. P. Keating and N.~C. Snaith, \textit{Random matrix theory and $\zeta(1/2+it)$}, Comm. Math. Phys. \textbf{214} (2000), 57--89 \medskip
  
\bibitem{lmfdb} The LMFDB Collaboration, The L-functions and Modular Forms Database, \url{http://lmfdb.warwick.ac.uk/data/zeros/zeta} \medskip

\bibitem{lune1} J. van de Lune and H. J. J. te Riele, \textit{On the zeros of the Riemann zeta function in the critical strip III}, Math. Comp. \textbf{41}  (1983), 759--767 \medskip

\bibitem{lune2} J. van de Lune et al., \textit{On the zeros of the Riemann zeta function in the critical strip IV}, Math. Comp. \textbf{46}  (1986), 667--681 \medskip

\bibitem{metha2} M. L. Metha and J. des Cloizeaux, \textit{The probability for several consecutive eigenvalues of a random matrix}, Ind. J. Pure Appl. Math. \textbf{3} (1972), 329--351 \medskip

\bibitem{metha3} M. L. Metha and J. des Cloizeaux, \textit{Asymptotic behavior of spacing distributions for the eigenvalues of random matrices}, J. Math. Phys. \textbf{14} (1973), 1648--1650 \medskip

\bibitem{metha1} M. L. Metha, \textit{Random Matrices}, 3rd edition, Elsevier Academic Press (2004) \medskip

\bibitem{montgomery}  H. L. Montgomery, \textit{The pair correlation of the zeros of the zeta function}, Proc. Symp. Pure Math. \textbf{24}, Amer. Math. Soc., Providence, R.I. (1973), 181--193 \medskip

\bibitem{odlyzko1} A. M. Odlyzko, \textit{On the distribution of spacings between zeros of the zeta function}, Math. Comp. \textbf{48} (1987), 273--308 \medskip

\bibitem{odlyzko2} A. M. Odlyzko, \textit{The $10^{20}$-th zero of the Riemann zeta function and 175 million of its neighbors}, available at \url{http://www.dtc.umn.edu/~odlyzko/unpublished/zeta.10to20.1992.pdf} (1992) \medskip
    
\bibitem{titchmarhs3} E. C. Titchmarsh, \textit{On van der Corput's method and the zeta-function of Riemann}, Quart. J. Math. \textbf{5} (1934), 98--105 \medskip

\bibitem{titchmarsh2} E. C. Titchmarsh, \textit{The zeros of the Riemann zeta function}, Proc. Roy. Soc. Ser. A \textbf{151} (1935), 234--255 \medskip

\bibitem{titchmarsh1} E. C. Titchmarsh, \textit{The Theory of the Riemann zeta-function}, 2nd edition, Oxford Science Publications, Oxford University Press, Oxford (1986) \medskip

\bibitem{trudgian1}  T. S. Trudgian, \textit{Further results on Gram's Law}, DPhil thesis, University of Oxford, Oxford (2009) \medskip

\bibitem{trudgian2} T. S. Trudgian, \textit{On the success and failure of Gram's Law and the Rosser Rule}, Acta. Arith. \textbf{143} (2011), 225--256 \medskip

\bibitem{weyl} H. Weyl, \textit{The Classical Groups}, Princeton Mathematica Series, Princeton University Press, Princeton (1946) \medskip

\bibitem{wilson} K. G. Wilson, \textit{Proof of a conjecture by Dyson}, J. Math. Phys. \textbf{3} (1962), 1040--1043 \medskip

\bibitem{zeilberger} D. Zeilberger, \textit{A combinatorial proof of Dyson's conjecture}, Discrete Math. \textbf{41} (1982), 317--332

\end{thebibliography}
\end{document}